\documentclass[11pt]{article}
\usepackage[british]{babel}
\usepackage{amsmath}
\usepackage{amssymb}
\usepackage{amsfonts}
\usepackage{amsthm}

\newtheorem{theorem}{Theorem}
\newtheorem{corollary}[theorem]{Corollary}
\newtheorem{proposition}[theorem]{Proposition}
\newtheorem{definition}[theorem]{Definition}

\allowdisplaybreaks[1]
\begin{document}
\thispagestyle{plain}

\begin{center}
	\textsc{\Large On the Entropy of a Two Step Random \\[1ex] Fibonacci Substitution} \vspace{1.5ex}
		
	\textsc{Johan Nilsson}\vspace{1.5ex}
	
	{\small Bielefeld University, Germany}
	
	{\small\texttt{jnilsson@math.uni-bielefeld.de}}
	
\end{center}

\begin{abstract}
We consider a random generalisation of the classical Fibonacci substitution. The substitution we consider is defined as the rule mapping $\mathtt{a}\mapsto \mathtt{baa}$ and $\mathtt{b} \mapsto \mathtt{ab}$ with probability $p$ and $\mathtt{b} \mapsto \mathtt{ba}$  with probability $1-p$ for $0<p<1$ and where the random rule is applied each time it acts on a $\mathtt{b}$. We show that the topological entropy of this object is given by the growth rate of the set of inflated random Fibonacci words, and we exactly calculate its value. 
\end{abstract}

{\small
\noindent MSC2010 classification; 68R15 Combinatorics on words, 05A16 Asymptotic enumeration, 37B10 Symbolic dynamics.
}

\section{Introduction}

In \cite{godreche} Godr\`eche and Luck define the random Fibonacci chain by the generalised substitution 
\[ 
	\theta : 
	\left\{
	\begin{array}{ll}
	  \mathtt{a} \mapsto \mathtt{b}   \\ 
	  \mathtt{b} \mapsto \begin{cases}
	      \mathtt{ab} & \textnormal{with probability $p$} \\
	      \mathtt{ba} & \textnormal{with probability $1-p$}
            \end{cases}
	\end{array}
	\right.
\]
for $0<p<1$ and where the random rule is applied each time $\theta$ acts on a $\mathtt{b}$. They introduce the random Fibonacci chain when studying quasi-crystalline structures and tilings in the plane. In their paper, it is claimed (without proof) that the topological entropy of the random Fibonacci chain is given by the growth rate of the set of inflated random Fibonacci words. This was later, with a combinatorial argument, proved in a more general context in \cite{nilsson}.

The renewed interest in this system, and in possible generalisations, stems from the observation that the natural geometric generalisation of the symbolic sequences by tilings of the line had to be Meyer sets with entropy and interesting spectra \cite{baake}. There is now a fair understanding of systems that emerge from the local mixture of inflation rules that each define the same hull. However, little is known so far about more general mixtures. Here we place our attention to one such generalisation. It is still derived from the Fibonacci rule, but mixes inflations that define distinct hulls.  

In this paper we consider the randomised substitution $\phi$ defined by
\[
\phi = \left\{
\begin{array}{ll}
  \mathtt{a} \mapsto \mathtt{baa}   \\ 
  \mathtt{b} \mapsto \begin{cases}
	      \mathtt{ab} & \textnormal{with probability $p$} \\
	      \mathtt{ba} & \textnormal{with probability $1-p$}
            \end{cases}
\end{array}
\right.
\]
for $0<p<1$ and where the random rule is applied each time $\phi$ acts on a $\mathtt{b}$. 
The substitution $\phi$ is a mixture of two substitutions, whose hulls are different. This is true, since the hull of the substitution $(\mathtt{a},\mathtt{b})\mapsto(\mathtt{baa},\mathtt{ab})$ contains words with the sub-words $\mathtt{aaa}$ and $\mathtt{bb}$, but neither of these sub-words are to be found in any word of the hull of $(\mathtt{a},\mathtt{b})\mapsto(\mathtt{baa},\mathtt{ba})$. For a more detailed survey of the differences and similarities of the generated hulls of these two substitutions see \cite{luck}.

Before we can state our main theorem in detail we need to introduce some notation. A word $w$ over an alphabet $\Sigma$ is a finite sequence $w_1w_2\ldots w_n$ of symbols from $\Sigma$. We let here $\Sigma = \{\mathtt{a},\mathtt{b}\}$. We denote a sub-word of $w$ by $w[a,b] = w_a w_{a+1}w_{a+2}\ldots w_{b-1}w_b$ and similarly we let $W[a,b] =\{w[a,b]: w\in W\}$. By $|\cdot|$ we mean the length of a word and the cardinality of a set. Note that $|w[a,b]| = b-a+1$. When indexing the brackets with a letter $\alpha$ from the alphabet, $|\cdot|_{\alpha}$, we shall mean the numbers of occurrences of $\alpha$ in the enclosed word. 

For two words $u = u_1u_2u_3\ldots u_n$ and $v = v_1v_2v_3\ldots v_m$ we denote by $uv$ the concatenation of the two words, that is, $uv = u_1u_2u_3\ldots u_n v_1 v_2 \ldots v_m$. Similarly we let for two sets of words $U$ and $V$ their product be the set $UV = \{uv: u\in U, v\in V\}$ containing all possible concatenations.

Letting $\phi$ act on the word $\mathtt{a}$ repeatedly yields an infinite sequence of words $r_n = \phi^{n-1}(\mathtt{a})$. We know that $r_1=\mathtt{a}$ and $r_2=\mathtt{baa}$. But $r_3$ is one of the words $\mathtt{abbaabaa}$ or $\mathtt{babaabaa}$ with probability $p$ or $1-p$. The sequence $\{r_n\}_{n=1}^{\infty}$ converges in distribution to an infinite random word $r$. We say that $r_n$ is an inflated word (under $\phi$) in generation $n$ and we introduce here sets that correspond to all inflated words in generation $n$; 

\begin{definition}
\label{def: recursive def An and Bn}
Let $A_1 = \{ \mathtt{a} \}$, $B_1 = \{ \mathtt{b} \}$ and for $n\geq2$ we define recursively 
\begin{align*}
    A_{n} &= B_{n-1}A_{n-1}A_{n-1}, \\
    B_{n} &= A_{n-1}B_{n-1} \cup B_{n-1}A_{n-1},   
\end{align*}
and we let $A := \lim_{n\to\infty}A_n$ and $B := \lim_{n\to\infty}B_n$.
\end{definition}

The sets $A$ and $B$ are indeed well defined. This is a direct consequence of Corollary \ref{cor: An prefix equal Bn prefix}. It is clear from the definition of $A_n$ and $B_n$ that all their elements have the same length, that is, for all $x,y \in A_n$ (or $x,y\in B_n$) we have $|x| = |y|$. By induction it easily follows that for $a\in A_n$ we have $|a|=f_{2n}$ and for $b\in B_n$ we have $|b|=f_{2n-1}$, where $f_m$ is the $m$th Fibonacci number, defined by $f_{n+1}=f_{n}+f_{n-1}$ with $f_0=0$ and $f_1=1$. 

For a word $w$ we say that $x$ is a sub-word of $w$ if there are two words $u,v$ such that $w= uxv$. The sub-word set $F(S,n)$ is the set of all sub-words of length $n$ of words in $S$. The \emph{combinatorial entropy} of the random Fibonacci chain is defined as the limit $\lim_{n\to\infty} \frac{1}{n}\log |F(A,n)|$. The combinatorial entropy is known to equal the topological entropy for our type of systems, see \cite{lind}. The existence of this limit is direct by Fekete's lemma \cite{fekete} since we have sub-additivity,  $\log|F(S,n+m)|\leq\log|F(S,n)|+\log|F(S,m)|$.
We can now state the main result in this paper.

\begin{theorem}
\label{thm: main}
The logarithm of the growth rate of the size of the set of inflated random Fibonacci words equals the topological entropy of the random Fibonacci chain, that is 
\begin{equation}
\label{eq: main}
\lim_{n\to\infty} \frac{\log |A_n|}{f_{2n}} 
=
\lim_{n\to\infty} \frac{\log |B_n|}{f_{2n-1}}
=
\lim_{n\to\infty} \frac{\log |F(C,n)|}{n} 
=
\frac{1}{\tau^3} \log 2,
\end{equation}
where $\tau$ is the golden mean, $\tau = \frac{1+\sqrt{5}}{2}$ and $C \in \{A,B\}$.
\end{theorem}

The outline of the paper is that we start by studying the sets $A_n$ and $B_n$. Next we give a finite method for finding the sub-word set $F(A,n)$, (which we will see is the same as $F(B,n)$). Thereafter we derive some diophantine properties of the Fibonacci number that will play a central part when we look at the distribution of the letters in words from $F(A,n)$. Finally we present an estimate of $|F(A,n)|$, leading up to the proof of Theorem \ref{thm: main}.

\section{Inflated words}

In this section we present the sets of inflated words and give an insight to their structure. The results presented here will also play an important role for the results in the coming sections.

\begin{proposition}
\label{prop: phi injectiv}
Let $u,v \in A_n$ (or both in $B_n$). Then $u\neq v$ if and only if $\{\phi(u)\}\cap \{\phi(v)\} = \emptyset$, where here $\{\phi(z)\}$ denotes the set of all possible words that can be obtained by applying $\phi$ on $z$.
\end{proposition}

\begin{proof}
Let $u\neq v$ and assume that $w\in \{\phi(u)\} \cap \{\phi(v)\}$. Denote by $\phi_u$ and $\phi_v$ the special choices of $\phi$ such that $w=\phi_u(u) =\phi_v(v)$. Let $k$ be the first position such that $u_k \neq v_k$ where $u=u_1 u_2\ldots u_m$ and $v=v_1 v_2\ldots v_m$. Then we may assume $u_k = \texttt{a}$ and $v_k=\texttt{b}$, otherwise just swap the names of $u$ and $v$.
Since we have $\phi(\texttt{a}) = \texttt{baa}$, we see that we must have $\phi_v(v_k)= \phi_v(\texttt{b}) = \texttt{ba}$. But then also $\phi_v(v_k v_{k+1}) = \phi_v(\texttt{bb}) = \texttt{baab}$. This then implies $u_{k+1} = \texttt{b}$, since if we have $u_{k+1} = \texttt{a}$ then there must be two consecutive $\texttt{a}$s in $w$ and we could not find a continuation in $v$. Hence we have $\phi_u(u_k u_{k+1}) = \phi_u(\texttt{ab}) = \texttt{baaba}$. As previously, $v$ must continue with a $\texttt{b}$. We now see that we are in a cycle, where $|\phi_u(u_k u_{k+1}\ldots u_{k+s})| = 3 + 2s$ and $|\phi_v(v_k v_{k+1}\ldots v_{k+s})| = 2(s+1)$. Since there is no $s \in \mathbb{N}$ such that we have $3 + 2s = 2(s+1)$, we conclude that there can be no such $w$. 
\end{proof}

We can now turn to the question of counting the elements in the sets $A_n$ and $B_n$. 

\begin{proposition}
\label{prop: size of An Bn}
For $n\geq 2$ we have 
\[	|A_n| = 2^{f_{2n-3}-1} \quad\textnormal{\textit{and}} \quad |B_n| = 2^{f_{2n-4}+1}.
\]
\end{proposition}

\begin{proof}
Let us start with the proof of the the size of $A_n$. From the \mbox{Definition \ref{def: recursive def An and Bn}} of $A_n$ and $B_n$ it follows by induction that $|x|_{\mathtt{b}} = f_{2n-2}$ for $x\in A_n$. Combining this with Proposition \ref{prop: phi injectiv} we find the recursion
\begin{equation}
\label{eq: size of an rec}
|A_n| 
= |A_{n-1}| \cdot 2^{|x|_{\mathtt{b}}}  
= |A_{n-1}| \cdot 2^{f_{2n-4}}. 
\end{equation}
The size of $A_n$ now follows from (\ref{eq: size of an rec}) by induction. For the size of $B_n$ we have, by the definition of $B_n$ and that we already know the size of $A_n$, 
\[	|B_{n}| = \frac{|A_{n+1}|}{|A_{n}||A_{n}|} 
= \frac{2^{f_{2n-1}-1}}{2^{f_{2n-3}-1}\cdot 2^{f_{2n-3}-1}} 
= 2^{f_{2n-4}+1},
\]
which completes the proof. 
\end{proof}

From Proposition \ref{prop: size of An Bn} the statements of the logarithmic limits of the sets $A_n$ and $B_n$ in Theorem \ref{thm: main} follows directly. Our next step is to give some result on sets of prefixes of $A_n$ and $B_n$. These results will play a central role when we later look at sets of sub-words.

\begin{proposition}
\label{prop: An sub An1, AnBn sub BnAn}
For $n\geq 2$ we have 
\begin{align}
A_n[1,f_{2n}-1] &\subset A_{n+1}[1,f_{2n}-1], \label{eq: An sub An1} \\
A_n[1,f_{2n}-1] &\subset \big(B_nA_n\big)[1,f_{2n}-1]. \label{eq: AnBn sub BnAn}
\end{align}
\end{proposition}

\begin{proof}
Let us first consider (\ref{eq: An sub An1}). We give a proof by induction on $n$. For the basis case, $n=2$, we have 
\begin{align*}
  A_2[1,f_{2\cdot2}-1] = A_2[1,2] =  \{\mathtt{ab}\} \subset\{\mathtt{ab},\mathtt{ba}\} = A_3[1,f_4-1].
\end{align*}
Now assume for induction that (\ref{eq: An sub An1}) holds for $2\leq n \leq p$. Then for $n=p+1$ we have by the induction assumption 
\begin{align*}
A_{p+1}[1,f_{2(p+1)}-1] 
&= \big( B_p A_p A_p \big)[1,f_{2(p+1)}-1] \\
&\subseteq \big( (A_p B_p \cup B_p A_p) A_p \big)[1,f_{2(p+1)}-1] \\
&= \big( B_{p+1} A_p \big)[1,f_{2(p+1)}-1] \\
&= B_{p+1} \big( A_p [1,f_{2p}-1] \big)\\
&\subset B_{p+1} \big( A_{p+1} [1,f_{2p}-1] \big)\\
&= \big(B_{p+1} A_{p+1}\big) [1,f_{2(p+1)}-1] \\
&= \big(B_{p+1} A_{p+1} A_{p+1} \big) [1,f_{2(p+1)}-1] \\
&= A_{p+2} [1,f_{2(p+1)}-1],
\end{align*}
which completes the induction and the proof of (\ref{eq: An sub An1}). Let us turn to the proof of (\ref{eq: AnBn sub BnAn}). By the help of (\ref{eq: An sub An1}) we have
\begin{align*}
A_n[1,f_{2n}-1] 
&= \big(B_{n-1}A_{n-1}A_{n-1}\big)[1,f_{2n}-1] \\
&= B_{n-1}A_{n-1}\big(A_{n-1}[1,f_{2(n-1)}-1]\big) \\
&\subset B_{n-1}A_{n-1}\big(A_{n}[1,f_{2(n-1)}-1]\big) \\
&= \big(B_{n-1}A_{n-1}A_{n}\big)[1,f_{2n}-1] \\
&\subseteq \big(B_{n}A_{n}\big)[1,f_{2n}-1]
\end{align*}
which concludes the proof.
\end{proof}

From Proposition \ref{prop: An sub An1, AnBn sub BnAn} it is straight forward, by recalling the recursive definition of $A_n$ and $B_n$, to derive the following equalities on prefix-sets.

\begin{corollary}
\label{cor: An prefix equal Bn prefix}
For $n\geq3$ we have
\begin{align*}
A_n[1,f_{2(n-1)}-1] &= A_{n+1}[1,f_{2(n-1)}-1], \\
B_n[1,f_{2(n-1)}-1] &= A_n[1,f_{2(n-1)}-1],     \\
B_n &= B_{n+1}[1,f_{2n-1}].                     
\end{align*}
\end{corollary}

We end the section by proving a result on suffixes of the sets $A_n$ and $B_n$ that we shall make use of in the next sections.

\begin{proposition}
\label{prop: An suf sub of Bn}
For $n\geq2$ we have
\begin{align}
  A_n[f_{2n-2}+2, f_{2n}] &\subseteq B_n[2,f_{2n-1}], \label{eq: An suf sub of Bn} \\
        B_n[2,f_{2n-1}] &= B_{n+1}[f_{2n}+2,f_{2n+1}]. \label{eq: Bn = Bn+1}
\end{align}
\end{proposition}

\begin{proof}
We give a proof by induction on $n$. For the basis case, $n=2$, we have
\[	A_2[f_{2}+2, f_4]= A_2[2,3] = \{\mathtt{a}\} \subseteq
\{\mathtt{a},\mathtt{b}\} = B_2[2,2].
\]
Now assume for induction that (\ref{eq: An suf sub of Bn}) holds for $2\leq n\leq p$. Then for the induction step, $n=p+1$, we have by the induction assumption
\begin{align*}
A_{p+1}[f_{2(p+1)-2}+2, f_{2(p+1)}]
&= \big(B_pA_pA_p\big)[f_{2(p+1)-2}+2, f_{2(p+1)}] \\
&= \big(A_pA_p\big)[f_{2p-2}+2, 2f_{2p}] \\
&= \big(A_p[f_{2p-2}+2, f_{2p}]\big)A_p \\
&\subseteq \big(B_p[2, f_{2p-1}]\big)A_p \\
&= \big(B_pA_p[2, f_{2p+1}]\big) \\
&\subseteq  B_{p+1}[2, f_{2(p+1)-1}],
\end{align*}
which completes the induction and the proof of (\ref{eq: An suf sub of Bn}).
For the proof of (\ref{eq: Bn = Bn+1}) we have 
\[	B_n[2,f_{2n-1}] 
= \big(A_nB_n\big)[f_{2n}+2,f_{2n}+f_{2n-1}]
\subseteq B_{n+1}[f_{2n}+2,f_{2n}+f_{2n-1}] 
\]
and for the converse inclusion we have by (\ref{eq: An suf sub of Bn})
\begin{align*}
      B_{n+1}[f_{2n}+2,f_{2n}+f_{2n-1}] 
      &= \big(A_nB_n \cup B_nA_n\big) [f_{2n}2,f_{2n}+f_{2n-1}] \\
      &= \big(B_n[2,f_{2n-1}]\big) \cup \big(A_n[f_{2n-2}+2, f_{2n}]\big) \\
      &\subseteq B_n[2,f_{2n-1}],
\end{align*}
which proves the equality (\ref{eq: Bn = Bn+1}).
\end{proof}

\section{Sets of sub-words}

Here we investigate properties of the sets of sub-words $F(A,m)$ and $F(B,m)$. We will prove that they coincide and moreover we show how to find them by considering finite sets, which will be central when estimating their size depending on $m$. 

First we turn our attention to proving that it is indifferent if we consider sub-words of $A_n$ or of $B_n$.

\begin{proposition}
\label{prop: FA eq FB}
For $n\geq 1$ we have 
\[	F(A_{n+1},f_{2n}-1) = F(B_{n+1},f_{2n}-1).
\]
\end{proposition}

\begin{proof}
Let us first turn to the proof of the inclusion 
\begin{equation}
\label{eq: FA sub of FB}
    F(A_{n+1},f_{2n}-1) \subseteq F(B_{n+1},f_{2n}-1).
\end{equation}
Let $x_{(k)} \in A_{n+1}[k,k-1+f_{2n}-1]$ for $1\leq k \leq f_{2n+1}+2$. It is clear that 
$x_{(k)}\in F(A_{n+1},f_{2n}-1)$ for any $k$. We have to prove that also $x_{(k)}\in F(B_{n+1},f_{2n}-1)$. 

For $1\leq k \leq f_{2n-1}+2$ we have 
\[	x_{(k)} \in F(B_nA_n,f_{2n}-1) \subseteq F(B_{n+1},f_{2n}-1).
\]

For $f_{2n-1}+3 \leq k \leq f_{2n}+1$ we have by Corollary \ref{cor: An prefix equal Bn prefix} that $x_{(k)}$ must be a sub-word of 
\begin{align*}
\big(A_nA_n\big)[3,f_{2n} + f_{2n-2}-1] 
&=  \big(A_nB_n\big)[3,f_{2n} + f_{2n-2}-1] \\
&=  B_{k+1}[3,f_{2n} + f_{2n-2}-1].
\end{align*}

For $f_{2n}+2 \leq k \leq f_{2n+1}+2$ we have By Proposition \ref{prop: An suf sub of Bn}
\begin{align*}
\big(B_nA_nA_n\big)[f_{2n}+2,f_{2n+2}] 
&=  \big(A_n[f_{2n-2}+2,f_{2n}]\big)A_n \\
&\subseteq  \big(B_n[2,f_{2n-1}]\big)A_n \\
&\subseteq  B_{n+1}[2,f_{2n+1}],
\end{align*}
which concludes the proof of the inclusion (\ref{eq: FA sub of FB}). For the converse inclusion it is enough to consider sub-words of $A_nB_n$, since any sub-word of $B_nA_n$ clearly is a sub-word of $A_{n+1}$. Therefore let $y_{(k)} \in (A_nB_n)[k,k-1+f_{2n}-1]$ for $1\leq k \leq f_{2n-1}+1$. We now proceed as in the case above. 

For $1\leq k \leq f_{2n-2}+1$ we have 
\begin{align*}
(A_nB_n)[1,f_{2n}+f_{2n-2}-1]
&=A_n\big(B_n[1,f_{2n-2}-1]\big) \\
&=A_n\big(A_n[1,f_{2n-2}-1]\big) \\
&=A_{n+1}[f_{2n+1}+1,f_{2n+1}+f_{2n-2}-1].
\end{align*}

For $f_{2n-2}+2\leq k \leq f_{2n-1}+2$ we have 
\begin{align*}
(A_nB_n)[f_{2n-2}+2,f_{2n-1}+2]
&= \big( A_n[f_{2n-2}+2,f_{2n}]\big)A_n \\
&= \big( B_n[2,f_{2n-1}]\big)A_n \\
&= A_{n+1}[2,f_{2n+1}],
\end{align*}
which completes the proof. 
\end{proof}

The above result shows that the set of sub-words from $A_n$ and $B_n$ coincide if the sub-words are not chosen too long. If, we consider the limit sets $A$ and $B$, their sets of sub-words turns out to be the same. We have the following

\begin{proposition}
\label{prop: FA = FB}
For $m\geq1$ we have $F(A,m) = F(B,m)$.
\end{proposition}

\begin{proof}
Let $x\in F(A,m)$. Then there is an $n$ such that 
\[	 x\in F(A_n,m) \subseteq F(A_{n}B_{n}\cup B_nA_n,m) = F(B_{n+1},m) \subseteq F(B,m).
\]
Similarly, if $x\in F(B,m)$. Then there is an $n$ such that 
\[	 x\in F(B_n,m) \subseteq F(B_nA_nA_n,m) = F(A_{n+1},m) \subseteq F(A,m),
\]
which completes the proof.
\end{proof}

The direct consequence of Proposition \ref{prop: FA = FB} is that we find the topological entropy in (\ref{eq: main}) independent if we look at sub-words from $A$ or $B$. 

Now, let us turn to the question of finding $F(A,m)$ from a finite set $A_n$ and not having to consider the infinite set $A$. 

\begin{proposition}
\label{prop: FAn1 = FAn2}
For $n\geq 2$ we have 
\[	F(A_{n+1},f_{2n}-f_{2n-3}) = F(A_{n+2},f_{2n}-f_{2n-3}).
\]
\end{proposition}

\begin{proof}
It is clear that $F(A_{n+1},f_{2n}-f_{2n-3}) \subseteq F(A_{n+2},f_{2n}-f_{2n-3})$ holds for all $n\geq2$. For the reverse inclusion assume that $x\in F(A_{n+2},f_{2n}-f_{2n-3})$. Note that we can write $A_{n+1}$ and $A_{n+2}$ on the form  
\begin{align}
  A_{n+1} &= B_n A_n A_n, \nonumber\\
  A_{n+2} &= B_n A_n B_n A_n A_n B_n A_n A_n \cup A_n B_n B_n A_n A_n B_n A_n A_n. 
  \label{eq: An2}
\end{align}
From we see (\ref{eq: An2}) that any $x$ is a sub-word of any element in some of the seven sets
\begin{equation}
\label{eq: seven sets}
\begin{array}{*{4}{l}}
A_nA_n,\phantom{A_n} & B_nA_n,    & A_nB_n,   & A_nB_nA_n, \\[1.3ex]
B_nB_n, & A_nB_nB_n, & B_nB_nA_n 
\end{array}
\end{equation}
in such a way that the first letter in $x$ is in the first factor (that is  $A_n$ or $B_n$) of the sets. If $x$ is a sub-word of $A_nA_n$ or $B_nA_n$ or completely contained in $A_n$ it is clear that we have $x\in F(A_{n+1},f_{2n}-f_{2n-3})$. For the case when $x$ is a sub-word of $A_nB_n$ it follows from Proposition \ref{prop: FA eq FB} that we have $x\in F(A_{n+1},f_{2n}-f_{2n-3})$. 

If $x$ is a sub-word of a word in  $A_nB_nA_n$ such that $x$ begins in the first $A_n$ factor and ends in the second. Then we have that $x$ is a sub-word of a word in the set
\begin{align*}
    \big(A_n[f_{2n-3}&+f_{2n-1}+2, f_{2n}]\big) B_{n-1}A_{n-1}\big(A_n[1, f_{2n-4}-1]\big) \\
    &= \big(A_{n}[f_{2n-3}+f_{2n-1}+2, f_{2n}]\big) B_{n-1}A_{n-1}\big(A_{n-1}[1, f_{2n-4}-1]\big) \\
    &= \big(A_{n}[f_{2n-3}+f_{2n-1}+2, f_{2n}]\big) \big(A_{n}[1, f_{2n-1}+f_{2n-4}-1]\big) \\
    &= \big(A_n A_n)[f_{2n-3}+f_{2n-1}+2, f_{2n}+f_{2n-1} + f_{2n-4}-1],
\end{align*}
and we see that we have $x\in F(A_{n+1},f_{2n}-f_{2n-3})$. 

If $x$ is a sub-word of a word in $B_nB_n$ we have, let us first consider the case when it is a sub-word of $B_nB_{n-1}A_{n-1}$. Then it follows that 
\[	B_nB_{n-1}A_{n-1} \subseteq B_n\big(A_n[1,f_{2n-1}]\big) = \big(B_nA_n\big)[1,2f_{2n-1}],
\]
so $x$ is a sub-word of a word in $A_{n+1}$. For the the second case, $B_nA_{n-1}B_{n-1}$, we have 
\begin{align*}
    B_nA_{n-1}B_{n-1} 
    &= A_{n-1}\big(B_{n-1}A_{n-1}\big)B_{n-1} \cup \big(B_{n-1}A_{n-1}A_{n-1}\big)B_{n-1} \\
    &\subseteq \big(A_{n}B_{n}A_{n}\big)[f_{2n-1}+1,3f_{2n-1}] \cup \big(A_{n}B_{n})[1,2f_{2n-1}],
\end{align*}
and again $x$ is a sub-word of a word in $A_{n+1}$, by what we just proved above. 

If $x$ is a sub-word of a word in $A_nB_nB_n$ we have by Corollary \ref{cor: An prefix equal Bn prefix},
\begin{align*}
    \big(A_nB_nB_n\big)&[f_{2n-1}+f_{2n-3}+1,2f_{2n}-f_{2n-3}-1] \\
    &= \big(A_n[f_{2n-1}+f_{2n-3}+1,f_{2n}]\big) B_n \big(B_n[1,f_{2n-4}-1]\big)  \\
    &= \big(A_n[f_{2n-1}+f_{2n-3}+1,f_{2n}]\big) B_n \big(A_n[1,f_{2n-4}-1]\big),
\end{align*}
which shows that $x$ is a sub-word of a word in $A_{n+1}$ by what we previously have shown. 

Finally, if $x$ is a sub-word of a word in $B_nB_nA_n$, we first consider the case when $x$ is a sub-word of a word in $B_nB_{n-1}A_{n-1}A_n$. By Corollary \ref{cor: An prefix equal Bn prefix} we have
\begin{align*}
    \big(B_nB_nA_n\big) & [2f_{2n-3} + 1,f_{2n+1}-f_{2n-3}-1] \\
    &= \big(B_n[2f_{2n-3}+1, f_{2n-1}]\big) B_{n-1}A_{n-1} \big(A_n[1,f_{2n-4}-1]\big)  \\
    &= \big(B_n[2f_{2n-3}+1, f_{2n-1}]\big) B_{n-1}A_{n-1} \big(A_{n-1}[1,f_{2n-4}-1]\big) \\
    &= \big(B_n[2f_{2n-3}+1, f_{2n-1}]\big) \big(A_{n}[1,f_{2n-1}+f_{2n-4}-1]\big),
\end{align*}
which by the help of the previous case shows that $x$ is a sub-word of a word in $A_{n+1}$. For last the case, $B_nA_{n-1}B_{n-1}A_n$, we have by Corollary \ref{cor: An prefix equal Bn prefix} and Proposition \ref{prop: An suf sub of Bn}, 
\begin{align*}
    \big(B_n&B_nA_n\big)  [2f_{2n-3} + 1,f_{2n+1}-f_{2n-3}-1] \\
    &= \big(B_n[2f_{2n-3}+1, f_{2n-1}]\big) A_{n-1}B_{n-1} \big(A_n[1,f_{2n-4}-1]\big)  \\
    &= \big(B_{n-1}[2f_{2n-3}-f_{2n-2}+1, f_{2n-3}]\big) A_{n-1}B_{n-1} \big(A_{n-1}[1,f_{2n-4}-1]\big)  \\
    &= \big(B_{n}[2f_{2n-3}-f_{2n-2}+1, f_{2n-1}]\big) \big(A_n[1,f_{2n-2}-1]\big),
\end{align*}
and again we see that $x$ is a sub-word of a word in $A_{n+1}$ by what we have proven above. 
\end{proof}

The result of Proposition \ref{prop: FAn1 = FAn2} can be extended to hold for sub-words from elements $A_n$ and $A_{n+k}$ where $k\geq 1$. A straight forward argument via induction gives 
\begin{equation}
\label{eq: FAn1 = FAnk}
F(A_{n+1},f_{2n}-f_{2n-3}) = F(A_{n+k},f_{2n}-f_{2n-3}) 
\end{equation}
for $k\geq1$. By combining Proposition \ref{prop: FAn1 = FAn2} and equation (\ref{eq: FAn1 = FAnk}) we can now prove that to find the factors set it is sufficient to only consider a finite set. 

\begin{proposition}
For $n\geq 2$ we have 
\begin{equation}
\label{eq: FAn1 = FA}
F(A_{n+1},f_{2n}-f_{2n-3}) = F(A,f_{2n}-f_{2n-3}).
\end{equation}
\end{proposition}
\begin{proof}
It is clear that we have $F(A_{n+1},f_{2n}-f_{2n-3}) \subseteq F(A,f_{2n}-f_{2n-3})$. For the reversed inclusion, let $x \in F(A,f_{2n}-f_{2n-3})$. Then there is a smallest $m\geq n+1$ such that $x$ is a sub-word of an element of $A_m$. Then (\ref{eq: FAn1 = FAnk}) gives
\[ x \in F(A_{m},f_{2n}-f_{2n-3}) = F(A_{n+1},f_{2n}-f_{2n-3}),
\]
which shows the desired inclusion.
\end{proof}

\section{Fibonacci numbers revisited}

In this section we shall restate, and adopt for our purpose, some of the Diophantine properties of the Fibonacci numbers, and use them to derive results on the distribution of the letters in the words in the sets $A_n$ and $B_n$. Let us introduce the notation
\[	\tau = \frac{1+\sqrt{5}}{2}\qquad\textnormal{and} \qquad \widehat{\tau} = \frac{1-\sqrt{5}}{2}
\]
for the roots of $x^2-x-1=0$. It is well known that $\tau$ and $\widehat{\tau}$ appears in Binet's formula  the Fibonacci numbers, see \cite{knuth}, 
\begin{equation}
\label{eq: fib explicit}
	f_n =  \frac{1}{\sqrt{5}} \left( \frac{1+\sqrt{5}}{2}\right)^n -
\frac{1}{\sqrt{5}} \left(\frac{1-\sqrt{5}}{2}\right)^n 
		=\frac{1}{\sqrt{5}}\left( \tau^n -\widehat{\tau}^{\,\,n}\right).
\end{equation}
From (\ref{eq: fib explicit}) it is with induction straight forward to derived 
\begin{equation}
\label{eq: fibonacci phi rec}
	f_{n} = \tau f_{n-1} + \widehat{\tau}^{\,\,n-1} = \tau^2 f_{n-2} + \widehat{\tau}^{\,\,n-2}.
\end{equation}

\begin{definition}
Let $\| \cdot \|$ denote the smallest distance to an integer. 
\end{definition}

By using the special property, $\tau^2 = \tau+1$ we have for an integer $k$ the following line of equalities 
\[
  \left\| \frac{1}{\tau^2}\, k \right\|
= \left\| \frac{\tau-1}{\tau}\, k \right\|
= \left\| k-\frac{1}{\tau}\, k \right\|
= \left\| \frac{1}{\tau}\, k \right\|
= \big\| (\tau-1) k \big\|
= \big\| \tau k \big\|.
\]
From (\ref{eq: fibonacci phi rec}) it follows that 
\[      \|\tau f_n\| =  \big\|f_{n+1} - \widehat{\tau}^{\,\,n} \big\| = \frac{1}{\tau^n}
\]
since $\widehat{\tau} = -\frac{1}{\tau}$. For an integer $k$, which is not a Fibonacci number we have the following estimate of how far away from an integer $\tau k$ is.

\begin{proposition}
\label{prop: phi k}
For a positive integer $k$ such that $f_{n-1} < k < f_n$ we have 
\begin{equation}
\label{eq: tau k}
    \|\tau k\| > \frac{1}{\tau^{n-2}}.
\end{equation}
\end{proposition}

\begin{proof}
We give a proof by induction on $n$. For the basis case $n=5$ the statement of the proposition follows by an easy calculation.
Now assume for induction that (\ref{eq: tau k}) holds for $5\leq n \leq p$. For the induction step, $n=p+1$, let $f_p < k < f_{p+1}$. Then, if $k-f_{p-1}$ is not a Fibonacci number we have
\[
\| \tau k \| 
=    \big\| \tau (k-f_{p}) + \tau f_{p} \big\| 
\geq \big\| \tau \underbrace{(k-f_{p})}_{< f_{p-1}}\big\| - \big\|\tau f_{p} \big\| 
>    \frac{1}{\tau^{p-3}} - \frac{1}{\tau^{p}} 
>    \frac{1}{\tau^{p-2}}.
\]
If $k-f_{p-1} = f_m$  for some $m<p-1$ then 
\[
\| \tau k \| 
\geq \| \tau f_m \| - \|\tau f_{p} \| 
=    \frac{1}{\tau^{m}} - \frac{1}{\tau^{p}} 
\geq \frac{1}{\tau^{p-2}} - \frac{1}{\tau^{p}} 
=    \frac{1}{\tau^{p-1}}. \qedhere
\]
\end{proof}

\begin{proposition}
\label{prop: nbr of b}
Let $x\in A_n[1,k]$ for $1\leq k\leq f_{2n}$ (or $x\in B_n[1,k]$ for $1\leq k\leq f_{2n-1}$) and $n\geq 2$. Then
\begin{equation}
\label{eq: nbr of b}
|x|_{\mathtt{b}} \in \left\{\left\lfloor \textstyle{\frac{1}{\tau^2}k} \right\rfloor, \left\lceil \textstyle{\frac{1}{\tau^2}k} \right\rceil\right\} 
\end{equation}
\end{proposition}

\begin{proof}
We give a proof by induction on $n$. The basis case, $n=2$, follows by considering each of the words contained in $A_2$ and $B_2$. To be able to use Proposition \ref{prop: phi k} in the induction step we have to consider the basis step $n=3$ as well, but only for the set $B_3$ (since the words in $A_2$ are of length $\geq 3$). This is however seen to hold by a straight forward enumeration of the elements of $B_3$. 

Now, assume for induction that (\ref{eq: nbr of b}) holds for $2\leq n \leq p$, for words both from $A_n$ and $B_n$. For the induction step, $n=p+1$, let us first derive an identity that we shall later on make use of.
Let $q$ and $m$ be positive integers such that $f_{m-1}< q < f_m$. Then, by the help of Proposition \ref{prop: phi k} we have
\begin{align}
\left\lfloor \frac{1}{\tau^2}(q-f_{m-1}) \right\rfloor 
&= \left\lfloor \frac{1}{\tau^2}q-f_{m-3} - \widehat{\tau}^{\,\,m-1} \right\rfloor \nonumber \\
&= \left\lfloor \frac{1}{\tau^2}q + \frac{(-1)^{m}}{\tau^{m-1}} \right\rfloor  - f_{m-3} \nonumber\\
&= \left\lfloor \frac{1}{\tau^2}q  \right\rfloor  - f_{m-3}. \label{eq: tau q}
\end{align}
With the same argumentation we can derive a similar result for $\lceil\cdot\rceil$. For the induction step we consider first the number of $\mathtt{b}$s in prefixes of words in $A_{p+1}=B_pA_pA_p$. It is clear from the induction assumption that (\ref{eq: nbr of b}) holds for $1\leq k \leq f_{2p-1}$. For $f_{2p-1}< k < f_{2p}$ or $f_{2p} < k < f_{2p+1}$ let $x=uv \in A_{p+1}[1,k]$ where $u\in B_p$. 
By the induction assumption we may assume that $|v|_{\mathtt{b}}$ is given by rounding downwards, (the result is obtained in a similar way for the case with $\lceil\cdot\rceil$). By (\ref{eq: tau q}) it now follows that 
\[
|uv|_{\mathtt{b}}
 = |u|_{\mathtt{b}} + |v|_{\mathtt{b}} 
 = f_{2p-3} + \left\lfloor \frac{1}{\tau^2}(k-f_{2p-1}) \right\rfloor 
 = \left\lfloor \frac{1}{\tau^2}k \right\rfloor.
\]
For $k=f_{2p}$ we have 
\[
	|uv|_{\mathtt{b}}
	= f_{2p-3} + \left\lfloor \frac{1}{\tau^2}(f_{2p}-f_{2p-1}) \right\rfloor 
	= \left\lfloor \frac{1}{\tau^2}f_{2p} + \frac{1}{\tau^{2p-1}} \right\rfloor 
	= \left\lceil \frac{1}{\tau^2}f_{2p} \right\rceil.
\]
For $f_{2p+1}< k < f_{2p+2}$ let $x=uvw \in A_{p+1}[1,k]$ where $u\in B_p$ and $v\in A_p$. Then the induction assumption and (\ref{eq: tau q}) gives.
\begin{align*}
	|uvw|_{\mathtt{b}}
	&= |u|_{\mathtt{b}} + |v|_{\mathtt{b}} + |w|_{\mathtt{b}} \\
	&= f_{2p-3} +f_{2p-2} + \left\lfloor \frac{1}{\tau^2}(k-f_{2p-1} - f_{2p}) \right\rfloor \\
	&= f_{2p-1} + \left\lfloor \frac{1}{\tau^2}(k-f_{2p+1})\right\rfloor \\
	&= \left\lfloor \frac{1}{\tau^2}k \right\rfloor  .
\end{align*}
For the last case $k=f_{2p+2}$ we have 
\[	|x|_{\mathtt{b}}
=\left\lfloor \frac{1}{\tau^2}(f_{2p+2}) \right\rfloor 
= \left\lfloor f_{2p} + \frac{1}{\tau^{2p}} \right\rfloor = f_{2p}.
\]
The case when we consider words from $B_{p+1}$ is treated in the same way, but where we don't need to do the induction step for the case $n=3$. This completes the induction and the proof.
\end{proof}

\begin{proposition}
\label{prop: bound nbr of b}
Let $x\in F(A_{n+2},f_{2n})$ for $n\geq2$. Then
\begin{equation}
\label{eq: bound nbr of b}
f_{2n-2} -1 \leq |x|_{\mathtt{b}} \leq f_{2n-2}+1.
\end{equation}
\end{proposition}

\begin{proof}
Let us first turn our attention to the upper bound in (\ref{eq: bound nbr of b}). In the same way as in the proof of Proposition \ref{prop: FAn1 = FAn2}, we consider sub-words of the seven sets, given in (\ref{eq: seven sets}). 

If $x$ is a sub-word, beginning at position $2 < k \leq f_{2n}$, in an element in $A_nA_n$ or $A_nB_n$ then
\begin{align*}
|x|_{\mathtt{b}} 
&\leq f_{2n-2} + \left\lceil \frac{1}{\tau^2}\big((k-f_{2n})+f_{2n}\big)\right\rceil - \left\lfloor \frac{1}{\tau^2}k \right\rfloor 
\leq f_{2n-2} +1.
\end{align*}
since the number of $\mathtt{b}$s in a word in $A_n$ is $f_{2n-2}$, and a word in $A_n$ is of length $f_{2n}$. 
The proof of the to the upper bound in (\ref{eq: bound nbr of b}), for the other sets in (\ref{eq: seven sets}) is obtained in the same way. 

For the lower bound we have
\begin{align*}
|x|_{\mathtt{b}} 
&\geq \left\lfloor \frac{1}{\tau^2}(k+f_{2n})\right\rfloor - \left\lceil \frac{1}{\tau^2}k \right\rceil \\
&=    f_{2n-2} +  \left\lfloor \frac{1}{\tau^2}k +\frac{1}{\tau^{2n-2}}\right\rfloor - \left\lceil \frac{1}{\tau^2}k \right\rceil \\
&\geq f_{2n-2} -1,
\end{align*}
for any $x\in F(A_{n+2},f_{2n})$.
\end{proof}

\section{Estimating the size of the sub-word set}

We shall in this section give an estimate of the sub-word  set $F(A,f_{2n})$ and give the final part of the proof of Theorem \ref{thm: main}. Let us introduce the set
\[	C_{n} = \phi\big( F(A, f_{2n-2}+1) \big).
\]
By Proposition \ref{prop: bound nbr of b} we can estimate the number of $\mathtt{b}$s in words in $F(A, f_{2n-2}+1)$. This estimate then gives that we have bounds on the length of words in $C_n$. That is, for $x\in C_n$ we have
\begin{equation}
	|x| = |x|_{\mathtt{a}} + |x|_{\mathtt{b}} \geq 3(f_{2n-3}-1) + 2(f_{2n-4}+2) = f_{2n}+1 
	\label{eq: low estim x in C}\\
\end{equation}
and
\begin{equation}
	|x| = |x|_{\mathtt{a}} + |x|_{\mathtt{b}} \leq 3(f_{2n-3}+2) + 2(f_{2n-4}-1) = f_{2n}+4.
	\label{eq: up estim x in C} 
\end{equation}

\begin{proposition}
For $n\geq 2$ we have 
\[      F(A, f_{2n}) = F\big(C_{n}, f_{2n}\big).
\]
\end{proposition}

\begin{proof}
The set $F\big(C_{n}, f_{2n}\big)$ is created by inflating words from $F(A,f_{2n-2}+1)$ which are then cut into suitable lengths. This implies that $F(A, f_{2n}) \supseteq F\big(C_{n}, f_{2n}\big)$. 

For the converse inclusion, let $x\in F(A, f_{2n})$. Then there is a word $w\in A_{n+1}$ and words $u,v$ such that $uxv\in A_{n+2}$ and $uxv\in\phi(w)$. For any word $z\in F\big(\{w\}, f_{2n-2}+1\big)$ we have from (\ref{eq: low estim x in C}) that any $s \in \phi(z)$ fulfils $f_{2n} + 1 \leq  |s|$. This gives that there is a word $z_x \in F\big(\{w\}, f_{2n-2}+1\big)$ such that $x$ is a sub-word of a word in $\phi(z_x)$, which implies $x\in F\big(C_{n}, f_{2n}\big)$.
\end{proof}

\begin{proposition}
\label{prop: size of Fn}
For $n\geq2$ we have
\begin{equation}
\label{eq: size of Fn}
      |F(A,f_{2n})| \leq  2^{f_{2n-3}+2n} \cdot 5^{n-1}.
\end{equation}
\end{proposition}
\begin{proof}
We give a proof by induction on $n$. For the basis case $n=2$ we have
\[	|F(A,f_4)| = 7 \leq 160 = 2^{f_1 + 4} \cdot 5.
\]
Assume for induction that (\ref{eq: size of Fn}) holds for $2\leq n \leq p$. For the induction step $n=p+1$, note that from (\ref{eq: low estim x in C}) and (\ref{eq: up estim x in C}) it follows that $|F(\{x\},f_{2p+2})| \leq 5$ for $x\in C_{p+1}$. By Proposition \ref{prop: bound nbr of b} we have that the number of $\mathtt{b}$s in $u\in F(A, f_{2p}+1)$ is at most $f_{2p-2}+2$.
This gives then, with the help of the induction assumption
\begin{align*}
\big|F\big(A,f_{2p+2}\big)\big| 
& \leq |C_{p+1}| \cdot 5 \\
& \leq \big|F\big(A,f_{2p}\big)\big|  \cdot 2^{f_{2p-2}+2} \cdot 5 \\
& \leq 2^{f_{2p-3}+2p} \cdot 5^{p-1} \cdot 2^{f_{2p-2}+2} \cdot 5 \\
&= 2^{f_{2(p+1)-3}+2(p+1)}\cdot 5^{p},
\end{align*}
which completes the proof. 
\end{proof}

We can now turn to proving the last equality in (\ref{eq: main}), and thereby completing the proof of Theorem \ref{thm: main}. By Proposition \ref{prop: size of Fn} we have
\begin{align*}
\lim_{n\to\infty} \frac{\log |F(A,f_{2n})|}{f_{2n}} 
&\leq \lim_{n\to\infty} \frac{\log \big(2^{f_{2n-3}+2n}\cdot 5^{n-1}\big) }{f_{2n}} \\
&= \lim_{n\to\infty} \frac{f_{2n-3}+2n}{f_{2n}} \log 2 + \frac{n-1}{f_{2n}}\log 5\\
&= \frac{1}{\tau^3} \log 2,
\end{align*}
which implies the equality in (\ref{eq: main}).

A further generalisation of the random Fibonacci substitutions, would be to study the structure occurring when mixing two substitutions with different inflation multipliers. This, however, seems to be a far more complex question.

\section{Acknowledgement}

The author wishes to thank M. Bakke and M. Moll at Bielefeld University, Germany, for our discussions and for reading drafts of the manuscript. 
This work was supported by the German Research Council (DFG), via CRC 701.


\begin{thebibliography}{99}
\small

\bibitem{baake}
M. Baake, M. Moll, Random noble means substitutions, 
In: Aperiodic Crystals, ed. S. Schmid et al., 
Springer, Dordrecht (2013), in press. \texttt{arXiv:1012.3462}.


\bibitem{fekete}
M. Fekete, \"Uber die Verteilung der Wurzeln bei gewissen algebraischen Gleich\-ungen mit. ganzzahligen Koeffizienten, 
\textit{Mathematische Zeitschrift} \textbf{17} (1923) 228--249.

\bibitem{godreche}
C. Godr\`eche, J. M. Luck, 
Quasiperiodicity and randomness in tilings of the plane,
\emph{J. Stat. Phys.} \textbf{55} (1989) 1--28.


\bibitem{knuth}
R. L. Graham, D. E. Knuth, O. Patashnik,
\textit{Concrete Mathematics},
Addison-Wesley Publishing Company, Reading, MA, (1994).

\bibitem{lind}
D. Lind, B Marcus, 
\textit{An Introduction to Symbolic Dynamics and Coding},
Cambridge University Press, Cambridge (1995).

\bibitem{luck}
J.M. Luck, C. Godr\`eche, A. Janner, T. Janssen,
The nature of the atomic surfaces of quasiperiodic self-similar structures,
\emph{J. Phys. A: Math. Gen.} \textbf{26} (1993), 1951--1999.

\bibitem{nilsson}
J. Nilsson,
On the entropy of a family of random substitutions,
\emph{Monatsh. Math.} \textbf{168} (2012) 563--577.
\texttt{arXiv:1103.4777}.

\end{thebibliography}
\end{document}